\documentclass[11 pt]{amsart}
\usepackage{amscd,amsfonts,amssymb,amsmath}
\usepackage{hyperref}
\usepackage{epsfig}
\newtheorem{theorem}{Theorem}[section]
\newtheorem{corollary}[theorem]{Corollary}
\newtheorem{lemma}[theorem]{Lemma}
\newtheorem{proposition}[theorem]{Proposition}
\theoremstyle{definition}

\newtheorem{question}[theorem]{Question}
\newtheorem{problem}[theorem]{Problem}

\newtheorem{remark}[theorem]{Remark}
\numberwithin{equation}{subsection}

\newcommand{\Alex}{\operatorname{Alex}}
\newcommand{\Aut}{\operatorname{Aut}}

\newcommand{\Conj}{\operatorname{Conj}}
\newcommand{\im}{\operatorname{Im}}
\newcommand{\Autcent}{\operatorname{Autcent}}
\newcommand{\Inn}{\operatorname{Inn}}

\newcommand{\R}{\operatorname{R}}
\newcommand{\Hom}{\operatorname{Hom}}
\newcommand{\GL}{\operatorname{GL}}
\newcommand{\Z}{\operatorname{Z}}

\newcommand{\id}{\mathrm{id}}

\begin{document}
\title{Automorphism groups of quandles arising from groups}
\author{Valeriy G. Bardakov}
\author{Pinka Dey}
\author{Mahender Singh}

\date{\today}
\address{Sobolev Institute of Mathematics and Novosibirsk State University, Novosibirsk 630090, Russia.}
\address{Laboratory of Quantum Topology, Chelyabinsk State University, Brat'ev Kashirinykh street 129, Chelyabinsk 454001, Russia.}
\address{Novosibirsk State Agrarian University, Dobrolyubova street, 160, Novosibirsk, 630039, Russia.}
\email{bardakov@math.nsc.ru}
\address{Indian Institute of Science Education and Research (IISER) Mohali, Sector 81,  S. A. S. Nagar, P. O. Manauli, Punjab 140306, India.}
\email{mahender@iisermohali.ac.in} 
\email{pinkadey@iisermohali.ac.in} 


\begin{abstract}
Let $G$ be a group and $\varphi \in \Aut(G)$. Then the set $G$ equipped with the binary operation $a*b=\varphi(ab^{-1})b$ gives a quandle structure on $G$, denoted by $\Alex(G, \varphi)$, and called the generalised Alexander quandle of $G$ with respect to $\varphi$. When $G$ is an additive abelian group and $\varphi = -\id_G$, then $\Alex(G, \varphi)$ is the well-known Takasaki quandle of $G$. In this paper,  we determine the group of automorphisms and inner automorphisms of Takasaki quandles of abelian groups with no 2-torsion, and Alexander quandles of finite abelian groups with respect to fixed-point free automorphisms. As an application, we prove that if $G\cong  (\mathbb{Z}/p \mathbb{Z})^n$ and $\varphi$ is multiplication by a non-trivial unit of $\mathbb{Z}/p \mathbb{Z}$, then $\Aut\big(\Alex(G, \varphi)\big)$ acts doubly transitively on $\Alex(G, \varphi)$. This generalises a recent result of \cite{Ferman} for quandles of prime order.

\end{abstract}
\maketitle

\section{Introduction}
Quandles were introduced independently by Matveev \cite{Matveev} and Joyce \cite{Joyce}, where Matveev referred them as distributive groupoids. To each oriented diagram $D_K$ of a tame knot $K$ in $\mathbb{R}^3$, they associated a quandle called the knot quandle $Q (K)$, which is presented by generators corresponding to the connected components of $D_K$  and relations $a*b=c$ whenever the arc $b$ passes over the double point separating arcs $a$ and $c$.  From the defining axioms of a quandle (see Section \ref{section2}) it follows that $Q(K)$ does not depend on the diagram $D_K$ and is an invariant of $K$.  In fact, Matveev and Joyce proved  that  it is a complete invariant  up to weak equivalence. More precisely, if $K$ and $K'$ are two knots with $Q(K) \cong Q(K')$, then $(\mathbb{R}^3, K)$ is homeomorphic to $(\mathbb{R}^3, K')$ when we ignore orientations of $\mathbb{R}^3$ and the knots. Over the years, quandles and their analogues have been investigated to construct  more  computable invariants for knots and links. See for example \cite{Carter, Kamada, Nelson} and references therein.

Although knot quandles are strong invariants of knots, it is usually difficult to distinguish two quandles from their presentations just like knot groups. However, a comparison of homomorphisms from two quandles to simpler quandles can provide useful information. For example, Fox's $n$-colourings are representations of the knot quandle into the dihedral quandle of order $n$ \cite{Fox, Przytycki}. In this paper, we attempt to make a unified investigation of automorphism groups and connectivity of quandles arising from groups. At this point, our approach is purely algebraic and we make no reference to knots. However, it would be interesting to explore implications of these results in knot theory.

The paper is organised as follows. In Section \ref{section2}, we recall the definition of a quandle and give important examples. In Section \ref{section3}, we consider the problem of embedding of quandles into conjugation quandles of groups. In Section \ref{section4}, we determine the group of automorphisms and inner automorphisms of Takasaki quandles of abelian groups with no 2-torsion (Theorem \ref{main-theorem}). We also determine a subgroup of the automorphism group of conjugation quandle of a group $G$ (Proposition \ref{aut-conj}). In Section \ref{section5}, we investigate commutativity and connectivity of generalised Alexander quandles. In particular, we prove that if the generalised Alexander quandle of a group with respect to an involutory central automorphism is connected, then the group is necessarily abelian (Theorem \ref{connnected-abelian}). Finally, in Section \ref{section6}, we determine the group of automorphisms and inner automorphisms of Alexander quandles of finite abelian groups with respect to fixed-point free automorphisms (Theorem \ref{fixed-point-free-theorem}). As an application, in Theorem \ref{FNT-generalisation}, we prove that if $G\cong  (\mathbb{Z}/p \mathbb{Z})^n$ and $\varphi$ is multiplication by a non-trivial unit of $\mathbb{Z}/p \mathbb{Z}$, then $\Aut\big(\Alex(G, \varphi)\big)$ acts doubly transitively on $\Alex(G, \varphi)$. This generalises a recent result of Ferman, Nowik and Teicher \cite{Ferman} for quandles of prime order.
\bigskip

\section{Quandle: Definition and examples}\label{section2}
A {\it quandle} is a set $X$ with a binary operation $(a,b) \mapsto a * b$ satisfying the following conditions:
\begin{enumerate}
\item$a*a=a$ for all $a \in X$;
\item For any $a,b \in X$ there is a unique $c \in X$ such that $a=c*b$;
\item $(a*b)*c=(a*c) * (b*c)$ for all $a,b,c \in X$.
\end{enumerate}

Besides knot quandles associated to knots, many interesting examples of quandles come from groups. Throughout the paper, we write arbitrary groups multiplicatively and abelian groups additively.
\begin{enumerate}
\item If $G$ is a group and $n \geq 1$ an integer, then the set $G$ equipped with the binary operation $a*b= b^{-n} a b^n$ gives a quandle structure on $G$. For $n=1$, it is called the {\it conjugation quandle}, and is denoted by $\Conj(G)$.
\item If $A$ is an additive abelian group, then the set $A$ equipped with the binary operation $a*b= 2b-a$ gives a quandle structure on $A$, denoted by $T(A)$. Such quandles (earlier called {\it Keis}) were first studied by Takasaki \cite{Takasaki}. For $A= \mathbb{Z}/n \mathbb{Z}$,  it is called the {\it dihedral quandle}, and is denoted by $\R_n$.
\item Let $A$ be an additive abelian group and $t \in \Aut(G)$. Then the set $A$ equipped with the binary operation $a* b=ta+(\id_A-t)b$ is a quandle called the {\it Alexander quandle} of $A$ with respect to $t$.
\item Next we consider quandles which are of utmost importance for this paper, and appeared first in the work of Joyce \cite{Joyce}. Let $G$ be a group and $\varphi \in \Aut(G)$, then the set $G$ with binary operation $a*b=\varphi(ab^{-1})b$ gives a quandle structure on $G$, which we denote by $\Alex(G, \varphi)$. These quandles are referred in the literature (for example \cite{Clark1, Clark}) as {\it generalized Alexander quandles}.
\end{enumerate}

Note that, if $A$ is an additive abelian group and $\varphi =-\id_A$, then $a*b=2b-a$. Thus we have  $\Alex(A, \varphi)=T(A)$, the Takasaki quandle of $A$.
On the other hand, for arbitrary $\varphi$, we have $a*b=\varphi(a) + (\id_A-\varphi)(b)$. Thus, in this case,  $\Alex(A, \varphi)$ is the usual Alexander quandle of $A$ with respect to $\varphi$.
\bigskip

\section{Embedding of quandles into conjugation quandles}\label{section3}
Let $X$ be a quandle. For each $x \in X$, the map $S_x: X \to X$ defined by $S_x(y)= y * x$ defines an automorphism of the quandle $X$, called an {\it inner automorphism}.  The group generated by all inner automorphisms of $X$ is denoted by $\Inn(X)$. It is evident that $\Inn(X)$ is a normal subgroup of $\Aut(X)$.

A quandle $X$ is called {\it involutory} if $S_x  S_x = \id_X$ for all $x \in X$. We say that an element $x \in X$ is {\it invertible} if there exists a $y \in X$ such that $S_x  S_y=S_y  S_x = \id_X$. In this case, we write $y=x^{-1}$. Clearly, in an involutory quandle, every element is invertible with $x= x^{-1}$ for all $x \in X$. An automorphism $\varphi$ of a group $G$ is called {\it fixed-point free} if $\varphi(a)=a$ implies that $a=1$. We first consider the following question.

\begin{question}
For which quandles $X$ does there exists a group $G$ such that $X$ embeds in the conjugation quandle $\Conj(G)$?
\end{question}

For special kind of Alexander quandles, we have the following answer.

\begin{proposition}
Let $G$ be an additive abelian group and $\varphi \in \Aut(G)$ be a fixed-point free involution. Let $X=\Alex(G, \varphi)$ be the Alexander quandle of $G$  with respect to $\varphi$. Then there is an embedding of quandles $X \hookrightarrow \Conj \big(\Inn(X)\big)$.
\end{proposition}

\begin{proof}
Define $S:X \to \Conj \big(\Inn(X)\big)$ by $S(a)= S_a$ for all $a \in X$. First note that
$$ S_a^2(b) = S_a\big(\varphi(b)-\varphi(a)+a\big)= b.$$
Hence $S_a$ is involutory. Further, $S(a_1 * a_2)= S\big( \varphi(a_1) - \varphi(a_2) + a_2\big) = S_{\varphi(a_1) - \varphi(a_2) + a_2}$ and
$$S_{\varphi(a_1) - \varphi(a_2) + a_2}(b)=\varphi(b-2a_2+a_1)+(2a_2-a_1).$$
 On the other hand, $S(a_1) * S(a_2)=S_{a_1} * S_{a_2}= S_{a_2}^{-1} S_{a_1}S_{a_2}$. Now for each $b \in X$, we have
\begin{eqnarray*}
S_{a_2}^{-1} S_{a_1}S_{a_2} (b) = S_{a_2} S_{a_1}S_{a_2} (b)&=&S_{a_2} S_{a_1}\big(\varphi(b)-\varphi(a_2)+ a_2\big)\\
& = &  S_{a_2}\big((\varphi(b)-\varphi(a_2)+ a_2)*a_1\big)\\
& = & S_{a_2}\big(b-a_2+ \varphi(a_2) -\varphi(a_1)+a_1\big)\\
& = & \big(b-a_2+ \varphi(a_2) -\varphi(a_1)+a_1\big)*a_2\\
& = & \varphi(b-2a_2+a_1)+(2a_2-a_1).
\end{eqnarray*}
Hence $S$ is a quandle homomorphism. Further, $S(a_1)=S(a_2)$ together with the hypothesis that $\varphi$ is fixed-point free implies that $a_1=a_2$. Hence $S$ is an embedding of $X$ in  $\Conj \big(\Inn(X)\big)$.
\end{proof}
\bigskip

\section{Automorphisms of quandles}\label{section4}
In what follows, we give a description of the  automorphism group of Takasaki quandle of an abelian group with no 2-torsion. First, we prove the following slightly more general result.

\begin{proposition}\label{joyce-embedding}
Let $G$ be a group and $\varphi \in \Aut(G)$. Then there is an embedding $\Z(G) \rtimes C_{\Aut(G)}(\varphi)  \hookrightarrow \Aut\big(\Alex(G, \varphi)\big)$, where $C_{\Aut(G)}(\varphi)$ is the centraliser of $\varphi$ in $\Aut(G)$ and $\Z(G)$ is the center of $G$.
\end{proposition}

\begin{proof}
Let $f \in C_{\Aut(G)}(\varphi)$. Then for $a, b \in G$, we have
\begin{eqnarray*}
f(a*b) &=& f\big( \varphi(a)\varphi(b)^{-1}b\big)\\
& = & f\big(\varphi(a)\big) f\big(\varphi(b)\big)^{-1}f(b)\\
& = & \varphi\big(f(a)\big) \varphi\big(f(b)\big)^{-1}f(b)\\
& = & f(a)*f(b).
\end{eqnarray*}
 Hence $C_{\Aut(G)}(\varphi) \leq  \Aut\big(\Alex(G, \varphi)\big)$.

Now, for $a \in \Z(G)$, let $t_a: G \to G$ be given by $t_a(b)=ba$ for all $b \in G$. Then $t_a \in \Aut\big(\Alex(G, \varphi)\big)$ and the map $a \mapsto t_a$ gives an embedding $\Z(G) \hookrightarrow \Aut\big(\Alex(G, \varphi)\big)$.

Finally, we show that the map $\Phi:\Z(G) \rtimes C_{\Aut(G)}(\varphi) \to \Aut\big(\Alex(G, \varphi)\big)$ given by $\Phi(a,f) = t_a  f$ is an embedding. The map is clearly injective. Further, for $(a_1,f_1), (a_2, f_2) \in  \Z(G) \rtimes C_{\Aut(G)}(\varphi)$, we have
\begin{eqnarray*}
\Phi\big((a_1,f_1)(a_2, f_2) \big) & = & \Phi\big((a_1 f_1(a_2), f_1  f_2) \big)\\
& = & t_{a_1 f_1(a_2)}  (f_1  f_2)\\
& = & (t_{a_1}  f_1)  (t_{a_2}  f_2)\\
& = & \Phi\big((a_1,f_1)\big)  \Phi\big((a_2, f_2) \big).
\end{eqnarray*}
This completes the proof of the proposition.
\end{proof}

Now, take $G$ to be an additive abelian group and $\varphi=-\id_G$. In this case $\Alex(G, \varphi)=T(G)$.

\begin{theorem}\label{main-theorem}
Let $G$ be an additive abelian group with no 2-torsion. Then the following hold:
\begin{enumerate}
\item $\Aut\big(T(G)\big) \cong G \rtimes \Aut(G)$.
\item $\Inn\big(T(G)\big) \cong 2G \rtimes \mathbb{Z}/2\mathbb{Z}$.
\end{enumerate}
\end{theorem}

\begin{proof}
We already proved in Proposition \ref{joyce-embedding} that there is an embedding $\Phi:G \rtimes \Aut(G)  \hookrightarrow \Aut\big(T(G)\big)$.  Let $f \in \Aut\big(T(G)\big)$. Then $f(a*b)=f(a)* f(b)$ implies that $f(2b-a)=2f(b)-f(a)$ for all $a, b \in T(G)$. Taking $a=2b$, we have  $f(2b)=2f(b)-f(0)$ for all $b \in G$, where $0$ is the identity of the group $G$.

Define $h: G \to G$ by $h(a)=f(a)-f(0)$ for all $a \in G$. We claim that $h \in \Aut(G)$. Since $f$ is a bijection, it follows that $h$ is also a bijection.  To show that $h$ is a group homomorphism, it suffices to show that $f(a+b)=f(a)+f(b)-f(0)$ for all $a, b \in G$. We can write
$$f(a+b)=f\big(2a-(a-b)\big)=2f(a)-f(a-b)$$
and
$$f(a+b)=f\big(2b-(b-a)\big)=2f(b)-f(b-a).$$
 Adding the two identities yields $2 f(a+b)= 2f(a)+ 2f(b)-f(a-b)-f(b-a)$. But, $f(a-b)=f\big(0-(b-a)\big)= 2f(0)-f(b-a)$. This gives $2 f(a+b)= 2f(a)+ 2f(b)- 2 f(0)$. Since $G$ has no 2-torsion, our claim follows. By definition of $h$, we can write $f=t_{f(0)}  h \in \Phi \big(G \rtimes \Aut(G) \big)$, and the proof of (1) is complete.

Let $r \in \Aut(G)$ be the automorphism given by $r(a) = -a$ for all $a \in G$. Then $\langle r \rangle \cong \mathbb{Z}/2\mathbb{Z}$. Note that $\Phi(2a, r) = t_{2a}  r=S_a$, and hence $\Phi$ embeds $2G \rtimes \mathbb{Z}/2\mathbb{Z}$ into $\Inn\big(T(G)\big)$. Conversely, any $S_a \in \Inn\big(T(G)\big)$ can be written as $S_a= t_{2a}  r$, and the proof of (2) is complete.
\end{proof}

As a consequence, we obtain the following.

\begin{corollary}
$\Aut\big( T(\mathbb{Z}^n)\big) \cong \mathbb{Z}^n \rtimes \GL(n, \mathbb{Z})$ and  $\Inn\big( T(\mathbb{Z}^n)\big) \cong (2\mathbb{Z})^n \rtimes \mathbb{Z}/2 \mathbb{Z}$ for all $n \geq 1$.
\end{corollary}

Recall that $T(\mathbb{Z}/n\mathbb{Z})=\R_n$, the dihedral quandle of cardinality $n$. We also retrieve the following result of Elhamdadi, Macquarrie and Restrepo \cite[Theorem 2.1 and 2.2]{Elhamdadi} for dihedral quandles of odd cardinalities.

\begin{corollary}
$\Aut( \R_n) \cong \mathbb{Z}/n\mathbb{Z} \rtimes (\mathbb{Z}/n\mathbb{Z})^{\times}$ and  $\Inn( \R_n) \cong \mathbb{Z}/n\mathbb{Z} \rtimes \mathbb{Z}/2 \mathbb{Z}$ for odd $n$, where $(\mathbb{Z}/n\mathbb{Z})^{\times}$ is the group of units of $\mathbb{Z}/n\mathbb{Z}$.
\end{corollary}

\begin{remark}
Hou proved a general result \cite[Proposition 2.1]{Hou} analogous to Theorem \ref{main-theorem} for Alexander quandles of groups under the condition that $G= \{\varphi(a)^{-1}a~|~a \in G \}$. This set is the so called {\it $\varphi$-twisted conjugacy class} of the identity element and is of independent interest. See \cite{Bardakov} for some related results. When $G$ is a finite group, this condition is equivalent to $\varphi$ being fixed-point free. In particular, when $G$ is a  finite additive abelian group and $\varphi = -\id_G$, then this condition is equivalent to $G$ being without 2-torsion. However, when $G$ is infinite, then the two conditions are entirely different. For example, $G=\mathbb{Q}/\mathbb{Z}$ satisfy $G=2G$, but has 2-torsion. On the other hand, $G=\mathbb{Z}$ is without 2-torsion, but does not satisfy $G=2G$. Thus, our result, though restricted to Takasaki quandles of abelian groups, can be considered as complementary to that of Hou's result.
\end{remark}

\begin{remark}
We later prove a stronger version of Theorem \ref{main-theorem} for finite abelian groups in Theorem \ref{fixed-point-free-theorem}.
\end{remark}

Next, we investigate the automorphism group of the conjugation quandle $\Conj(G)$ of a group $G$. First, note that 
$$\Inn \big(\Conj(G)\big) \cong \Inn(G) \cong G/\Z(G)$$ for any group $G$. Our main observation is the following.

\begin{proposition}\label{aut-conj}
Let $G$ be a group with center $\Z(G)$. Then there is an embedding of groups $\Z(G) \rtimes \Aut(G) \hookrightarrow \Aut\big(\Conj(G)\big)$.
\end{proposition}

\begin{proof}
For each $a \in \Z(G)$, let $t_a: G \to G$ be given by $t_a(b)=ba$ for all $b \in G$. Then $t_a \in \Aut\big(\Conj(G))$ and the map $a \mapsto t_a$ gives an embedding $G \hookrightarrow \Aut\big(\Conj(G)\big)$.  Obviously, $\Aut(G)\subseteq  \Aut\big(\Conj(G)\big)$.

It remains to show that the map $\Phi: \Z(G) \rtimes \Aut(G) \to \Aut\big(\Conj(G)\big)$ given by $\Phi(a,f) = t_a  f$ is an embedding. The map is clearly injective. Further, for $(a_1,f_1), (a_2, f_2) \in  \Z(G) \rtimes \Aut(A)$, we have
\begin{eqnarray*}
\Phi\big((a_1,f_1)(a_2, f_2) \big) & = & \Phi\big((a_1f_1(a_2), f_1  f_2) \big)\\
& = & t_{a_1f_1(a_2)}  (f_1  f_2)\\
& = & (t_{a_1}  f_1)  (t_{a_2}  f_2)\\
& = & \Phi\big((a_1,f_1)\big)  \Phi\big((a_2, f_2) \big).
\end{eqnarray*}
This completes the proof of the proposition.
\end{proof}

It is known that $\Aut\big(\Conj(\Sigma_3)\big) \cong \Inn\big(\Conj(\Sigma_3)\big) \cong \Sigma_3$ (see, for example, \cite{Elhamdadi}). Also, it is well-known that if $n \ge 3$ and $n \neq 6$, then $\Aut\big(\Sigma_n \big) \cong \Inn\big(\Sigma_n  \big) \cong \Sigma_n$. In view of Proposition \ref{aut-conj} and the preceding discussion, the following problem seems interesting.

\begin{problem}
Classify groups $G$ for which $\Aut\big(\Conj(G)\big) \cong \Z(G) \rtimes \Aut(G)$.
\end{problem}

We suspect that symmetric groups have such automorphism groups.
\bigskip

\section{Commutativity and connectivity of Alexander quandles}\label{section5}
Let $G$ be a group acting on a set $X$ from left. For each $x \in X$, the subgroup $G_x= \{g \in G~|~gx=x \}$ is called the { \it isotropy subgroup} at $x$. For each $1 \leq k \leq |X|$, we say that $G$ acts {\it $k$-transitively} on $X$ if for each pair of $k$-tuples $(x_1,\dots,x_k)$ and $(y_1,\dots,y_k)$ of distinct elements of $X$, there exists $g \in G$ such that $gx_i = y_i$ for each $1 \leq i \leq k$. A 1-transitive action is simply called {\it transitive} and a 2-transitive action is also called {\it doubly transitive}.

Let $X$ be a quandle and $\Inn(X)$ be its inner automorphism group. Then $\Inn(X)$ acts on $X$ by $(f, x) \mapsto f(x)$. For $k\geq 1$, the quandle $X$ is called {\it $k$-transitive quandle} if $\Inn(X)$ acts $k$-transitively on $X$. A 1-transitive quandle is also called {\it connected quandle}. The term was coined by Joyce, who also proved that knot quandles are connected \cite[Corollary 15.3]{Joyce}. Similarly, a 2-transitive quandle is also called {\it two-point homogeneous quandle}. Two-point homogeneous quandles are of fundamental importance, and a complete classification of finite two-point homogeneous quandles was obtained very recently \cite{Vendramin, Wada}.

A quandle $X$ is called {\it commutative} if $x * y = y * x$ for all $x,y \in X$. It must be noted that, this definition of commutativity is different from that of Joyce \cite{Joyce} and Neumann \cite{Neumann}.

Recall that, given a group $G$ and $\varphi \in \Aut(G)$, the generalised Alexander quandle $\Alex(G, \varphi)$ of $G$ with respect to $\varphi$ is the set $G$ with binary operation $a*b=\varphi(ab^{-1})b$ for $a, b \in G$.

\begin{proposition}
Let $G$ be a group and $\varphi \in \Aut(G)$. If $\Alex(G, \varphi)$ is commutative, then $\varphi(a^2)=a$ for all $a \in G$. Further, the converse holds if and only if $G$ is abelian.
\end{proposition}

\begin{proof}
If $\Alex(G, \varphi)$ is commutative, then for $a, b \in G$, we have $\varphi(a)\varphi(b)^{-1}b=\varphi(b)\varphi(a)^{-1}a$. Taking $b=1$ yields $\varphi(a^2)=a$.

Conversely, if $\varphi(a^2)=a$ for all $a \in G$, then $a*b=b*a$ if and only if $\varphi(ab)= \varphi(ba)$. This happens if and only if $G$ is abelian.
\end{proof}

If $G$ is an additive abelian group, then $\Alex(G, \varphi)$ is the usual Alexander quandle. As a consequence, we recover the following result of Bae and Choi \cite[Theorem 2.8]{Bae}.

\begin{corollary}
Let $G$ be an additive abelian group and $\varphi \in \Aut(G)$. Then the  Alexander quandle $\Alex(G, \varphi)$ is commutative if and only if $2\varphi = \id_G$.
\end{corollary}

An automorphism $\varphi$ of a group $G$ is called {\it central} if $a^{-1}\varphi(a) \in \Z(G)$ for all $a \in G$. Note that these are precisely the automorphisms of $G$ which induce identity on the central quotient $G/\Z(G)$. The set of all central automorphisms is a subgroup of $\Aut(G)$, and denoted by $\Autcent(G)$. Note that $\Autcent(G)=\Aut(G)$ for an abelian group $G$. However, there exist non-abelian groups $G$ for which $\Autcent(G)=\Aut(G)$. See Curran and McCaughan \cite{Curran} for detailed references. 

\begin{lemma}\label{central-auto}
Let $G$ be a group and $\varphi \in \Autcent(G)$. Then the following hold:
\begin{enumerate}
\item The map $\tilde{\varphi}:G \to \Z(G)$ defined by $\tilde{\varphi}(a)=a^{-1}\varphi(a)$ is a group homomorphism. 
\item $\varphi \mapsto \tilde{\varphi}$ gives an injective map $\Autcent(G) \to \Hom\big(G, \Z(G)\big)$.
\item If $\varphi$ is fixed-point free, then $G$ is abelian.
\end{enumerate}
\end{lemma}

\begin{proof}
For $a, b \in G$, we have
\begin{eqnarray*}
\tilde{\varphi}(ab) & = & (ab)^{-1}\varphi(ab)\\
&=& b^{-1}a^{-1}\varphi(a) \varphi(b)\\
&=& a^{-1}\varphi(a) b^{-1} \varphi(b),~\textrm{since}~a^{-1}\varphi(a) \in \Z(G)\\
&=& \tilde{\varphi}(a)\tilde{\varphi}(b).
\end{eqnarray*}
It is easy to see that the map $\varphi \mapsto \tilde{\varphi}$ is injective. By (1), $\tilde{\varphi}$ is a homomorphism. Further, note that $\varphi$ is fixed-point free if and only if  $\tilde{\varphi}$ is injective. Thus, $\tilde{\varphi}:G \to \Z(G)$ is an embedding and $G$ is abelian.
\end{proof}

Regarding connectivity of generalised Alexander quandles we prove the following.

\begin{theorem}\label{connnected-abelian}
Let $G$ be a group and $\varphi \in \Autcent(G)$ be an involution. If $\Alex(G, \varphi)$ is connected, then $G$ is abelian.
\end{theorem}

\begin{proof}
Let $\varphi^2= \id_G$. Then for all $g, a \in G$, we have
$$S_g  S_g (a)=S_g \big (\varphi(a) \varphi(g^{-1}) g\big)=\varphi^2(a) \varphi^2(g^{-1}) \varphi(g) \varphi(g^{-1}) g=a.$$
Hence $\Alex(G, \varphi)$ is involutory, and $S_g=S_g^{-1}$ for all $g \in G$. Let $\Alex(G, \varphi)$ be connected and  $b \in G$. Then there exists a $f \in \Inn(G)$ such that $b^{-1}= f(1)$, where $f= S_{g_k}  \cdots  S_{g_1}$ for some $g_i \in G$. Then

\begin{eqnarray*}
b^{-1} & = & S_{g_k}  \cdots   S_{g_1}(1)\\
&=& (\cdots ((1 * g_1) * g_2) *\cdots ) *g_k\\
&=& (\cdots(\varphi(g_1^{-1}) g_1)* g_2) *\cdots ) *g_k\\
&=& \varphi^k(g_1^{-1}) \varphi^{k-1}(g_1) \varphi^{k-1}(g_2^{-1}) \varphi^{k-2}(g_2) \varphi^{k-2}(g_3^{-1}) \cdots \varphi^2(g_{k-1}^{-1}) \varphi(g_{k-1}) \varphi(g_k^{-1}) g_k\\
&=& \varphi^k(g_1^{-1}) \varphi^{k-1}(g_1) \varphi^{k-1}(g_2^{-1}) \varphi^{k-2}(g_2) \varphi^{k-2}(g_3^{-1}) \cdots \varphi^2(g_{k-1}^{-1})   \varphi(g_k^{-1}) g_k \varphi(g_{k-1}),\\
&& \textrm{since}~a^{-1}\varphi(a) \in \Z(G)~\textrm{for all}~ a \in G\\
&=& \varphi^k(g_1^{-1}) \varphi^{k-1}(g_1) \varphi^{k-1}(g_2^{-1}) \varphi^{k-2}(g_2) \varphi^{k-2}(g_3^{-1}) \cdots \varphi \big(\varphi(g_{k-1}^{-1})  g_k^{-1} \big) \big(\varphi(g_{k-1}^{-1}) g_k^{-1} \big)^{-1}.
\end{eqnarray*}

By repeatedly using the above property of $\varphi$, we get 
$$b^{-1} =\varphi(a)a^{-1},~ \textrm{where}~ a=\varphi^{k-1}(g_1^{-1})\varphi^{k-2}(g_2^{-1}) \cdots \varphi(g_{k-1}^{-1})g_k^{-1}.$$
 This gives $b= a \varphi(a^{-1})=\tilde{\varphi}(a^{-1}) \in \Z(G)$. This proves $G=\Z(G)$.
\end{proof}

In view of \cite[Proposition 3.3]{Singh}, under the hypothesis of the preceding theorem, we have $\Alex(G, \varphi)$ is connected if and only if $G$ is abelian and $\tilde{\varphi}$ is surjective. We need the following well-known characterisation of finite connected Alexander quandles. See, for example, \cite[Theorem 2.5]{Bae} or \cite[Corollary 7.2]{Hulpke}.

\begin{theorem}\label{Bae-Choe-theorem}
Let $G$ be a finite additive abelian group and $\varphi \in \Aut(G)$. Then the following are equivalent:
\begin{enumerate}
\item  $\Alex(G, \varphi)$ is connected.
\item $\varphi$ is fixed-point free.
\item $\tilde{\varphi} \in \Aut(G)$.
\end{enumerate}
\end{theorem}

An immediate consequence of the preceding theorem is that; the dihedral quandle $\R_n$ is connected if and only if $n$ is odd.

\begin{remark}\label{purely-non-ab-def}
A group is said to be {\it purely non-abelian} if it does not have any non-trivial abelian direct factor. Adney and Yen \cite[Theorem 1]{Adney} proved that; if $G$ is a finite purely non-abelian group, then the map $\varphi \mapsto \tilde{\varphi}$ is a bijection. Thus, for a finite purely non-abelian group $G$, in view of Theorem \ref{connnected-abelian}, the set $\Hom\big(G, \Z(G)\big)=\Hom\big(G/G', \Z(G)\big)$ gives central automorphisms for which the Alexander quandle structures on $G$ are not connected. For example, if $Q_8$ is the Quaternion group of order 8, then $|\Autcent(Q_8)|=|\Hom\big(Q_8, \mathbb{Z}/2\mathbb{Z}\big)|=|\Hom\big(\mathbb{Z}/2\mathbb{Z} \oplus \mathbb{Z}/2\mathbb{Z}, \mathbb{Z}/2\mathbb{Z}\big)|=4$.
\end{remark}
\bigskip

\section{Doubly transitive action}\label{section6}
In this section, we generalise a recent result of Ferman, Nowik and Teicher \cite{Ferman} about doubly transitive action of the automorphism group of a finite Alexander quandle of prime order on the underlying quandle. An alternate proof of the same result was given by Watanabe \cite{Watanabe}. Our approach is more general than \cite{Ferman} and \cite{Watanabe}.  Let 
$$\Aut\big(\Alex(G, \varphi)\big)_0=\big\{f \in  \Aut\big(\Alex(G, \varphi)\big)~|~ f(0)=0 \big\}$$
 be the isotropy subgroup of $\Aut\big(\Alex(G, \varphi)\big)$ at $0$. We first prove a stronger version of Theorem \ref{main-theorem} for finite abelian groups with a fixed-point free automorphism.

\begin{theorem}\label{fixed-point-free-theorem}
Let $G$ be a finite additive abelian group and $\varphi \in \Aut(G)$ be a fixed-point free automorphism. Then the following hold:
\begin{enumerate}
\item $C_{\Aut(G)}(\varphi)= \Aut\big(\Alex(G, \varphi)\big)_0$.
\item $ \Aut\big(\Alex(G, \varphi)\big) \cong G \rtimes C_{\Aut(G)}(\varphi)$.
\item $ \Inn\big(\Alex(G, \varphi)\big) \cong G \rtimes \langle \varphi \rangle$.
\end{enumerate}
\end{theorem}

\begin{proof}
Let $f \in \Aut\big(\Alex(G, \varphi)\big)_0$. Then for $a, b \in G$, we have
$$f\big(\varphi(a)-\varphi(b)+b \big)=\varphi\big(f(a)\big)-\varphi\big(f(b)\big)+f(b).$$
Taking $a=0$ gives $f\big(-\varphi(b)+b \big)=-\varphi\big(f(b)\big)+f(b)$ and taking $b=0$ gives $f\big(\varphi(a)\big)=\varphi\big(f(a)\big)$. Putting in the above equation gives
$$f\big(\varphi(a)-\varphi(b)+b \big)=f\big(\varphi(a)\big)+f\big(-\varphi(b)+b \big).$$
Since $\varphi$ is fixed-point free, by Theorem \ref{Bae-Choe-theorem}, $\tilde{\varphi} \in \Aut(G)$. Hence $f(x+y)=f(x)+f(y)$ for all $x, y \in G$ and $f \in C_{\Aut(G)}(\varphi)$. Conversely, if $f \in C_{\Aut(G)}(\varphi)$, then $f(0)=0$ and $f(a*b)=f(a)*f(b)$, and the proof of (1) is complete.

By Proposition \ref{joyce-embedding}, there is an embedding $\Phi:G \rtimes C_{\Aut(G)}(\varphi)  \hookrightarrow \Aut\big(\Alex(G, \varphi)\big)$. Let $f \in \Aut\big(\Alex(G, \varphi)\big)$. Define $h: G \to G$ by $h(a)=f(a)-f(0)$ for all $a \in G$. Then $h \in \Aut\big(\Alex(G, \varphi)\big)$ and $h(0)=0$. Hence $h \in C_{\Aut(G)}(\varphi)$ and $f= t_{f(0)} h \in \Phi\big(G \rtimes C_{\Aut(G)}(\varphi) \big)$. This proves the assertion (2).

Since $\varphi$ is fixed-point free, by Theorem \ref{Bae-Choe-theorem}, $G= \im(\tilde{\varphi})= \{- \varphi(a)+a~|~a \in G \}$. For each $a \in G$, $S_a=t_{- \varphi(a)+a} \varphi \in \Phi(G \rtimes \langle \varphi \rangle)$. Conversely, each $a \in G$ can be written as $a=- \varphi(a')+a'$ for some $a' \in G$. Then we have $t_a \varphi=S_{a'}$, and the proof of assertion (3) is complete.
\end{proof}

The result below is well-known and easy to prove.

\begin{lemma}\label{doubly-transitive-criteria}
Let $G$ be a group acting on a set $X$ and $x \in X$. Then the action is doubly transitive if and only if it is transitive and the isotropy subgroup $G_x$ acts transitively on $X\setminus \{x\}$.
\end{lemma}

The following is probably well-known, and we provide a proof here for the sake of completeness.

\begin{lemma}\label{transitive-criteria}
Let $G$ be a non-trivial finite group. Then $\Aut(G)$ acts transitively on $G \setminus \{ 0\}$ if and only if $G \cong (\mathbb{Z}/p \mathbb{Z})^n$ for some prime $p$ and integer $n \geq 1$.
\end{lemma}

\begin{proof}
Suppose that $\Aut(G)$ acts transitively on $G \setminus \{ 0\}$. For each prime $p$ dividing order of $G$, there exists an element $a \in G$ of order $p$ by Cauchy's Theorem for finite groups. Given any $b \in G \setminus \{ 0\}$, there exists $\varphi \in \Aut(G)$ such that $\varphi(a)=b$. This implies that all elements of $G \setminus \{ 0\}$ have the same order $p$. Since $G$ is a non-trivial $p$-group, its center $\Z(G)$ is non-trivial. For each $a \in G \setminus \{ 0\}$ and $b \in \Z(G) \setminus \{ 0\}$, there exists $\varphi \in \Aut(G)$ such that $\varphi(b)=a$. Since  $\Z(G)$ is invariant under each automorphism of $G$, we have $a\in \Z(G) \setminus \{ 0\}$. Hence $G\cong  (\mathbb{Z}/p \mathbb{Z})^n$ for some integer $n \geq 1$. The converse is obvious, and hence omitted.
\end{proof}

We can now prove our final result, which generalises \cite[Theorem 3.9]{Ferman} to elementary abelian $p$-groups.

\begin{theorem}\label{FNT-generalisation}
Let $G=  (\mathbb{Z}/p \mathbb{Z})^n$ for some prime $p$ and integer $n \geq 1$. If $\varphi$ is multiplication by a non-trivial unit of $\mathbb{Z}/p \mathbb{Z}$, then $\Aut\big(\Alex(G, \varphi)\big)$ acts doubly transitively on $\Alex(G, \varphi)$.
\end{theorem}

\begin{proof}
Since $\varphi$ is fixed-point free, by Theorem \ref{Bae-Choe-theorem}, $\Alex(G, \varphi)$ is connected. This implies that $\Inn\big(\Alex(G, \varphi)\big)$, and hence $\Aut\big(\Alex(G, \varphi)\big)$ acts transitively on $\Alex(G, \varphi)$. In view of Lemma \ref{doubly-transitive-criteria}, it suffices to prove that $\Aut\big(\Alex(G, \varphi)\big)_0$ acts transitively on $\Alex(G, \varphi) \setminus \{ 0 \}$. By Theorem \ref{fixed-point-free-theorem},  $\Aut\big(\Alex(G, \varphi)\big)_0=C_{\Aut(G)}(\varphi)$. But, $C_{\Aut(G)}(\varphi)= \Aut(G)$, which by Lemma \ref{transitive-criteria} acts transitively on $\Alex(G, \varphi) \setminus \{ 0 \}$.
\end{proof}

\begin{remark}
It is worth pointing out that $\Alex(G, \varphi)$ is not two-point homogeneous for $G =  (\mathbb{Z}/p \mathbb{Z})^n$ with $n \geq 2$. This follows from recent results of Vendramin \cite[Theorem 3]{Vendramin} and Wada \cite[Corollary 4.5]{Wada}, where they classify all finite two-point homogeneous quandles. In fact, they proved that any such quandle is isomorphic to an Alexander quandle defined by primitive roots over a finite field.
\end{remark}

\begin{remark}
McCarron \cite[Proposition 5]{McCarron} proved that; if $k \geq 2$ and $X$ is a finite $k$-transitive quandle with at least four elements, then $k=2$. Further, the dihedral quandle with three elements $R_3$ is the unique 3-transitive quandle. Thus, higher order transitivity does not exist in quandles with at least four elements.
\end{remark}

In view of the above remark, the following problem seems natural.

\begin{problem}
For an integer $k \geq 2$, classify all finite quandles $X$ for which $\Aut(X)$ acts $k$-transitively on $X$.
\end{problem}

\medskip \noindent \textbf{Acknowledgement.} The authors thank the referee for many useful comments. The work was supported by the DST-RSF Project INT/RUS/RSF/P-2 (Grant of Russian
Science Foundation No.16-41-02006).

\end{document}